\documentclass[a4paper,12pt]{article}

\usepackage[T1]{fontenc}
\usepackage[utf8]{inputenc}
\usepackage[english]{babel}
\usepackage[inline,shortlabels]{enumitem}
\usepackage{amsmath}
\usepackage{amssymb}
\usepackage{amsthm}
\usepackage[noadjust]{cite}
\usepackage[a4paper,total={6in,9in}]{geometry}
\usepackage[colorlinks,citecolor=blue]{hyperref}
\usepackage[capitalize,noabbrev,nameinlink]{cleveref}
\usepackage{xcolor}
\usepackage{longtable}
\usepackage{tikz,pgfplots,pgfplotstable}
\usetikzlibrary{patterns}
\usepackage{algorithmic}
\usepackage{marginnote}
\usepackage[labelformat=simple]{subcaption}

\setlist{font=\normalfont,topsep=1ex,parsep=0ex}

\setlist[enumerate]{label=(\alph*)}

\numberwithin{equation}{section}
\numberwithin{table}{section}    
\numberwithin{figure}{section}

\crefname{figure}{Figure}{Figures}
\crefname{table}{Table}{Tables}
\crefname{assumption}{Assumption}{Assumptions}
\Crefname{ALC@unique}{Step}{Steps}

\newlist{alglist}{enumerate}{1}
\setlist[alglist]{topsep=1ex,parsep=0ex,leftmargin=*,label=\textbf{Step~\arabic*.}}

\newcommand{\R}{\mathbb{R}}

\newcommand\norm[1]{\left\Vert#1\right\Vert}
\newcommand{\N}{\mathbb{N}}

\newcommand{\dom}{\operatorname{dom}}

\newtheoremstyle{bolddef}{}{}{\normalfont}{}{\bfseries}{.}{ }{\thmname{#1}\thmnumber{ #2}\thmnote{ (#3)}}

\newtheoremstyle{boldplain}{}{}{\itshape}{}{\bfseries}{.}{ }{\thmname{#1}\thmnumber{ #2}\thmnote{ (#3)}}

\theoremstyle{bolddef}
\newtheorem{definition}{Definition}[section]
\newtheorem{algorithm}[definition]{Algorithm}
\newtheorem{assumption}[definition]{Assumption}

\newtheorem{remark}[definition]{Remark}

\theoremstyle{boldplain}
\newtheorem{lemma}[definition]{Lemma}
\newtheorem{theorem}[definition]{Theorem}
\newtheorem{proposition}[definition]{Proposition}
\newtheorem{corollary}[definition]{Corollary}

\newlength\figureheight
\newlength\figurewidth

\pgfplotsset{width=7cm,compat=1.3}

\definecolor{todocolor}{rgb}{1.0,0.0,0.0}

\newcommand\email[1]{\href{mailto:#1}{\texttt{#1}}}

\newcommand{\orcid}[1]{ORCID: \href{https://orcid.org/#1}{#1}}

\newcommand{\mscLink}[1]{\href{http://www.ams.org/mathscinet/msc/msc2020.html?t=#1}{#1}}

\begin{document}

\title{
	\bfseries\scshape 
	Convergence Properties of Monotone and Nonmonotone Proximal Gradient Methods Revisited
	%\thanks{bla}
	}

\date{\today}

\author{Christian Kanzow
	\thanks{%
		University of Würzburg,
		Institute of Mathematics,
		97074 Würzburg,
		Germany,
		\email{kanzow@mathematik.uni-wuerzburg.de},
		\orcid{0000-0003-2897-2509}
	}
	\hspace*{-4mm} \and \hspace*{-4mm}
	Patrick Mehlitz
	\thanks{%
		Brandenburgische Technische Universität Cottbus-Senftenberg,
		Institute of Mathematics,
		03046 Cottbus,
		Germany,
		\email{mehlitz@b-tu.de},
		\orcid{0000-0002-9355-850X},
		University of Mannheim,
		School of Business Informatics and Mathematics,
		68159 Mannheim,
		Germany
	} 
}

\maketitle

{
\small\textbf{\abstractname.}
	Composite optimization problems, where the sum of a smooth and a merely lower semicontinuous
	function has to be minimized, are often tackled numerically by means of proximal gradient
	methods as soon as the lower semicontinuous part of the objective function is of simple enough
	structure. The available convergence theory associated with these methods (mostly) requires the
	derivative of the smooth part of the objective function to be (globally) Lipschitz continuous,
	and this might be a restrictive assumption in some practically relevant scenarios.
	In this paper, we readdress this classical topic and provide convergence results for
	the classical (monotone) proximal gradient method and one of its nonmonotone extensions which
	are applicable in the absence of (strong) Lipschitz assumptions.
		This is possible since, for the price of forgoing convergence rates, we omit 
		the use of descent-type lemmas in our analysis.
	
\par\addvspace{\baselineskip}
}

{
\par\addvspace{\baselineskip}
}

{
\par\addvspace{\baselineskip}
}

\section{Introduction}\label{Sec:Intro}

In this paper, we address the classical problem of minimizing the sum
of a smooth function $f$ and a nonsmooth function $\phi$, 
also known under the name \emph{composite optimization}. 
This setting received much attention throughout 
the last years due to its inherent practical relevance in, e.g., machine learning, 
data compression, matrix completion, and image processing, see e.g.\ 
\cite{BianChen2015,BrucksteinDonohoElad2009,Chartrand2007,DiLorenzoLiuzziRinaldiSchoenSciandrone2012,LiuDaiMa2015,MarjanovicSolo2012}.

A standard technique for the solution of composite optimization problems is the 
proximal gradient method, introduced by Fukushima and Mine \cite{FukushimaMine1981}
and popularized e.g.\ by Combettes and Wajs in \cite{CombettesWajs2005}.
	A particular instance of this method is the celebrated
	iterative shrinkage/threshold algorithm (ISTA), see, e.g.\ \cite{BeckTeboulle2009}. 
A summary of existing
results for the case where the nonsmooth term $\phi$ is defined by a convex function
is given in the monograph by Beck \cite{Beck2017}.

The proximal gradient method can also be interpreted as a forward-backward splitting method, 
	see \cite{Bruck1977,Passty1979} for its origins 
	and \cite{BauschkeCombettes2011} for a modern view, 
and is able to handle problems where the 
nonsmooth term $\phi$ is given by a merely lower semicontinuous function, see, e.g.\ the seminal works 
\cite{AttouchBolteSvaiter2013,BolteSabachTeboulle2014}. 
These references also provide convergence and rate-of-convergence results
by using the 
	popular \emph{descent lemma} together with the celebrated 
	\emph{Kurdyka--\L ojasiewicz property}.

To the best of our knowledge, however, 
	the majority of available
convergence results for proximal gradient methods
assume that the smooth term $f$ is continuously differentiable 
with a globally Lipschitz continuous
gradient (or they require local Lipschitzness together with a bounded level set which,
again, implies the global Lipschitz continuity on this level set).
This requirement,
	which is the essential ingredient for the classical descent lemma,
is often satisfied for
standard applications of the proximal gradient method in data science and image processing, where
$f$ appears to be a quadratic function.

In this paper, we aim to get rid of this global Lipschitz condition. This is motivated by
the fact that the algorithmic application we have in mind
does not satisfy this Lipschitz property since the smooth term $f$ corresponds to the
augmented Lagrangian function of a general nonlinear constrained optimization problem, 
which rarely has a 
globally Lipschitz continuous gradient or a bounded level set. The proximal gradient method
will be used to solve the resulting subproblems which forces us to generalize the
convergence theory up to reasonable assumptions which are likely to hold in our framework. 
	We refer the interested reader to 
	\cite{ChenGuoLuYe2017,GuoDeng2021,JiaKanzowMehlitzWachsmuth2021,DeMarchiJiaKanzowMehlitz2022}
	where such augmented Lagrangian \emph{proximal} methods are investigated.

Numerically, a nonmonotone version of the proximal gradient method is often preferred.
Based on ideas by Grippo et al.\ \cite{GrippoLamparielloLucidi1986} in the context of
smooth unconstrained optimization problems, Wright et al.\ \cite{WrightNowakFigueiredo2009}
developed a nonmonotone proximal gradient method for composite optimization problems
known under the name SpaRSA. 
	In their paper, the authors assume that the nonsmooth part $\phi$ of the objective function
	is convex.	
	Almost simultaneously, the authors of \cite{BirginMartinezRaydan2000} presented 
	a nonmonotone projected gradient method for the minimization of a differentiable
	function over a convex set.
	Their findings can be interpreted as a special case of the results from
	\cite{WrightNowakFigueiredo2009} where $\phi$
	equals the indicator of a convex set.
The ideas from \cite{BirginMartinezRaydan2000,WrightNowakFigueiredo2009} 
were subsequently generalized in
the papers \cite{ChenGuoLuYe2017,ChenLuPong2016} where the proximal gradient method
is used as a subproblem solver within an augmented Lagrangian and penalization scheme,
respectively. However, the authors did not address the aforementioned problematic lack of Lipschitzness
in these papers which causes their convergence theory to be barely applicable in their algorithmic framework.
In \cite{LiLin2015,WangLiu2021}, the authors present nonmonotone extensions of ISTA which can handle merely
lower semicontinuous terms in the objective function. Again, for the convergence analysis, global Lipschitzness
of the smooth term's derivative is assumed.
Due to its practical importance, we therefore aim to provide a convergence theory for the nonmonotone proximal gradient 
method without using any Lipschitz assumption. 

In the seminal paper \cite{BauschkeBolteTeboulle2017},
the authors consider the composite optimization problem with
both terms being convex, but without a global Lipschitz
assumption for the gradient of the smooth part $f$. They get suitable
rate-of-convergence results for the iterates generated
by a Bregman-type proximal gradient method using only
a local Lipschitz condition. In addition, however, they require that
there is a constant $ L > 0 $ such that $ L h - f $ is
convex, where $ h $ is a convex function which defines the
Bregman distance (in our setting, $ h $ equals the squared norm). 
Some examples indicate that this convexity-type condition
is satisfied in many practically relevant situations. Subsequently, this
approach was generalized to the nonconvex setting in 
\cite{BolteSabachTeboulleVaisbourd2018} using, once again,
a local Lipschitz assumption only, as well as the slighty
stronger assumption (in order to deal with the nonconvexity)
that there exist $ L > 0 $ and a convex function $ h $ such
that both $ L h - f $ and $ L h + f $ are convex. Note that
the constant $ L $ plays a central role in the design of the
corresponding proximal-type methods. Particularly, it is used
explicitly for the choice of stepsizes. Finally, the very recent 
paper \cite{CohenHallakTeboulle2022} proves global convergence
results under a local Lipschitz assumption (without the
additional convexity-type condition), but assumes that the
iterates and stepsizes of the underlying proximal gradient method remain bounded.

To the best of our knowledge, this is the current state-of-the-art
regarding the convergence properties of proximal gradient methods.
The aim of this paper is slightly different, since we do not
provide rate-of-convergence results, but conditions which
guarantee accumulation points to be suitable stationary points
of the composite optimization problem. This is the essential feature of the proximal gradient method which, for example, is exploited in \cite{ChenGuoLuYe2017,JiaKanzowMehlitzWachsmuth2021,DeMarchiJiaKanzowMehlitz2022} to develop augmented Lagrangian proximal methods.
We also stress that, in this particular situation, the 
above assumption that $ L h \pm f $ is convex for some 
$ L > 0 $ is often violated unless we are dealing with linear
constraints only.

Our analysis does not 
require a global Lipschitz assumption and is not based
on the crucial descent lemma,  
contrasting \cite{BauschkeBolteTeboulle2017,BolteSabachTeboulleVaisbourd2018} mentioned above.
The results show that we can get stationary accumulation points only
under a local Lipschitz assumption and, depending on the
properties of $ \phi $, sometimes even without any
Lipschitz condition. In any case, a convexity-type condition
like $ L h \pm f $ being convex for some constant $ L $ is
not required at all. Moreover, the implementation of our
proximal gradient method does not need any knowledge of
the size of any Lipschitz-type constant.

Since the aim of this paper is to get a better understanding of the theoretical convergence
properties of both monotone and nonmonotone proximal gradient methods, and since these methods
have already been applied numerically to a large variety of problems, we do not include any
numerical results in this paper. 

Let us recall that we are mainly interested in conditions ensuring that accumulation points
of sequences produced by the proximal gradient method are stationary.
The main contributions of this paper show that this property 
holds (neglecting a few
technical conditions) for the monotone proximal gradient method if either the smooth function 
$f$ is continuously differentiable and the nonsmooth function $\phi$ is continuous on its domain
	(e.g., this assumption holds for a constrained optimization problem 
	where $\phi$ corresponds to the indicator function of a nonempty and closed set),
or if $f$ is differentiable with a locally Lipschitz continuous derivative and
$\phi$ is an arbitrary lower semicontinuous function. Corresponding statements 
for the nonmonotone proximal gradient method require stronger assumptions, basically the
uniform continuity of the objective function
	on a level set.
	That, however, is a standard assumption in the literature dealing with
	nonmonotone stepsize rules.

The paper is organized as follows: In \cref{Sec:Prelims}, we give a detailed
statement of the composite optimization problem and provide some necessary background
material from variational analysis. The convergence properties of the monotone and
nonmonotone proximal gradient method are then discussed in \cref{Sec:GenSpecGrad,Sec:GenSpecGradNM}, respectively. 
We close with some final remarks in \cref{Sec:Final}.

\section{Problem Setting and Preliminaries}\label{Sec:Prelims}

We consider the \emph{composite} optimization problem
\begin{equation}\label{Eq:P}\tag{P}
   \min_x \ \psi(x):=f(x) + \phi (x), \quad \quad x \in \mathbb X,
\end{equation}
where $ f\colon \mathbb{X} \to \mathbb{R} $ is
continuously differentiable, $ \phi\colon \mathbb{X} \to \overline{\mathbb{R}}:= \R\cup\{\infty\} $ 
is lower semicontinuous (possibly infinite-valued and nondifferentiable), 
and $ \mathbb{X} $ denotes a
Euclidean space, i.e., a real and finite-dimensional Hilbert space.
We assume that the domain $\dom\phi:=\{x\in\mathbb X\,|\,\phi(x)<\infty\}$ of $ \phi $
is nonempty to rule out trivial situations.
	In order to minimize the function $\psi\colon\mathbb X\to\overline\R$ in
	\eqref{Eq:P}, it seems reasonable to exploit the composite structure of $\psi$,
	i.e., to rely on the differentiability of $f$ on the one hand, and on some beneficial
	structural properties of $\phi$ on the other one. This is the idea behind splitting methods.

Throughout the paper, the Euclidean space $\mathbb X$ 
will be equipped with the inner product $\langle\cdot,\cdot\rangle\colon \mathbb X\times\mathbb X\to\R$
and the associated norm $\norm{\cdot}$.
For some set $A\subset\mathbb X$ and some point $x\in\mathbb X$, we make use of
$A+x=x+A:=\{x+a\,|\,a\in A\}$ for the purpose of simplicity.
For some sequence $\{x^k\}\subset\mathbb X$ and $x\in\mathbb X$, $x^k\to_\phi x$ means that
$x^k\to x$ and $\phi(x^k)\to\phi(x)$.
The continuous linear operator
$f'(x)\colon\mathbb X\to\R$ denotes the derivative of $f$ at $x\in\mathbb X$, and we will
make use of $\nabla f(x):=f'(x)^*1$ where $f'(x)^*\colon\R\to\mathbb X$ is the adjoint of $f'(x)$.
This way, $\nabla f$ is a mapping from $\mathbb X$ to $\mathbb X$.
Furthermore, we find $f'(x)d=\langle\nabla f(x),d\rangle$ for each $d\in\mathbb X$.
 
The following concepts are standard in variational analysis, see e.g.\ \cite{Mordukhovich2018,RockafellarWets2009}.
Let us fix some point $x\in\dom\phi$. Then
\[
	\widehat\partial\phi(x)
	:=
	\left\{
		\eta\in\mathbb X\,\middle|\,
		\liminf\limits_{y\to x}\frac{\phi(y)-\phi(x)-\langle \eta,y-x\rangle}{\norm{y-x}}\geq 0
	\right\}
\] 
is called the {\em regular} (or {\em Fr\'{e}chet}) {\em subdifferential} of $\phi$ at $x$.
Furthermore, the set
\[
	\partial\phi(x)
	:=
	\left\{
		\eta\in\mathbb X\,\middle|\,
		\exists\{x^k\},\{\eta^k\}\subset\mathbb X\colon x^k\to_\phi x,\,\eta^k\to\eta,\,\eta^k\in\widehat{\partial}\phi(x^k)\,\forall k\in\N
	\right\}
\]
is well known as the {\em limiting} (or {\em Mordukhovich}) 
{\em subdifferential} of $\phi$ at $x$.
Clearly, we always have $\widehat{\partial}\phi(x)\subset\partial\phi(x)$ by construction.
Whenever $\phi$ is convex, equality holds, and both subdifferentials coincide with the
subdifferential of convex analysis, i.e.,
\[
	\widehat{\partial}\phi(x)
	=
	\partial\phi(x)
	=
	\{
		\eta\in\mathbb X\,|\,\forall y\in\mathbb X\colon\,\phi(y)\geq\phi(x)+\langle\eta,y-x\rangle
	\}
\]
holds in this situation.
It can be seen right from the definition that whenever $x^*\in\dom\phi$ is a local minimizer
of $\phi$, then $0\in\widehat\partial\phi(x^*)$, which is referred to as Fermat's rule,
see \cite[Proposition~1.30(i)]{Mordukhovich2018}.

Given $ x \in \dom\phi $, the limiting subdifferential has the important robustness
property 
\begin{equation}\label{Eq:robustness}
   \left\{ 
   	\eta\in\mathbb X\,\middle|\,
   	\exists\{x^k\},\{\eta^k\}\subset\mathbb X\colon\,x^k\to_\phi x,\,\eta^k\to\eta,\,\eta^k\in\partial\phi(x^k)\,\forall k\in\N
   \right\}
   \subset
   \partial\phi(x),
\end{equation}
see \cite[Proposition~1.20]{Mordukhovich2018}. 
Clearly, the converse inclusion $\supset$ is also valid by definition of the limiting subdifferential.
Note that in situations where $\phi$ is discontinuous at $x$, the requirement $x^k\to_\phi x$ in the definition
of the set on the left-hand side in \eqref{Eq:robustness} is strictly necessary. 
In fact, the usual outer semicontinuity in the sense of set-valued mappings, given by
\begin{equation}\label{Eq:osc}
   \left\{ 
   	\eta\in\mathbb X\,\middle|\,
   	\exists\{x^k\},\{\eta^k\}\subset\mathbb X\colon\,x^k\to x,\,\eta^k\to\eta,\,\eta^k\in\partial\phi(x^k)\,\forall k\in\N
   \right\}
   \subset
   \partial\phi(x),
\end{equation}
would be a much stronger condition in this situation and does not hold in general.

Whenever $x\in\dom\phi$ is fixed, the sum rule
\begin{equation}
		\label{eq:regular_sum_rule}
		\widehat{\partial}(f+\phi)(x)
		=
		\nabla f(x)+\widehat{\partial}\phi(x)
\end{equation}
holds, see \cite[Proposition~1.30(ii)]{Mordukhovich2018}.
Thus, due to Fermat's rule, whenever $x^*\in\dom\phi$ is a local minimizer of $f+\phi$, we have
$0\in \nabla f (x^*)+\widehat{\partial}\phi(x^*)$.
This condition is potentially more restrictive than
$0\in \nabla f (x^*)+\partial\phi(x^*)$ which, naturally, also serves as a necessary optimality
condition for \eqref{Eq:P}. However, the latter is more interesting from an
algorithmic point of view as it is well known from the literature on splitting methods
comprising nonconvex functions $\phi$. If $\phi$ is convex, there is no difference between
those stationarity conditions.

Throughout the paper, a point $x^*\in\dom\phi$ satisfying
$0\in \nabla f (x^*)+\partial\phi(x^*)$ will be called a 
{\em Mordukhovich-stationary} ({\em M-stationary} for short) point of 
\eqref{Eq:P} due to the appearance of the limiting subdifferential. In the literature, the name 
\emph{limiting critical point} is used as well. We close this section with two special instances
of problem \eqref{Eq:P} and comment on the corresponding M-stationary conditions.

\begin{remark}\label{Rem:constrained_opt}
Consider the constrained optimization problem
\[
   \min_x \ f(x) \quad \text{subject to} \quad x \in C
\]
for a continuously differentiable function $ f\colon\mathbb X\to\R $ and a nonempty and closed 
(not necessarily convex) set $ C \subset \mathbb{X} $. This problem is equivalent 
to the unconstrained problem \eqref{Eq:P} by setting $ \phi  := \delta_C $, 
where $ \delta_C \colon \mathbb X\to\overline{\R} $ denotes the indicator function of the set $C$, vanishing
on $C$ and taking the value $\infty$ on $\mathbb X\setminus C$, which is lower semicontinuous due to the 
assumptions regarding $ C $.
The corresponding M-stationarity condition is given by
\[
   0 \in \nabla f (x^*) + \partial \delta_C (x^*) = \nabla f (x^*) + \mathcal N_C (x^*),
\]
where $\mathcal N_C(x^*) $ denotes the \emph{limiting} (or \emph{Mordukhovich}) \emph{normal cone}, 
see \cite[Proposition~1.19]{Mordukhovich2018}.
\end{remark}

\begin{remark}\label{rem:non_Lipschitz_constrained_optimization}
Consider the more general constrained optimization problem
\[
   \min_x \ f(x) + \varphi (x) \quad \text{subject to} \quad x \in C
\]
with $ f\colon\mathbb X\to\R$ and $ C\subset\mathbb X $ as in \Cref{Rem:constrained_opt},
and $ \varphi\colon\mathbb X\to\overline{\R} $ being another
lower semicontinuous function (which might represent a regularization, penalty, or sparsity-promoting term, for
example). Setting $ \phi := \varphi + \delta_C $, we obtain
once again an optimization problem of the form \eqref{Eq:P}. The corresponding
M-stationarity condition is given by
\[
   0 \in \nabla f (x^*) + \partial \phi (x^*) = \nabla f (x^*) + \partial ( \varphi + \delta_C ) (x^*).
\]
Unfortunately, the sum rule   
\[
	\partial ( \varphi + \delta_C ) (x^*) 
	\subset
	\partial \varphi (x^*) + \partial \delta_C (x^*) 
	=
	\partial \varphi (x^*) + \mathcal N_C(x^*)
\] 
does not hold in general. However,
for locally Lipschitz functions $ \varphi $, for example, it applies,
see \cite[Theorems~1.22, 2.19]{Mordukhovich2018}.
Note that the resulting stationarity condition
\[
   0 \in \nabla f (x^*) + \partial \varphi (x^*) + \mathcal N_C (x^*)
\]
might be slightly weaker than M-stationarity as introduced above.
Related discussions can be found in \cite[Section~3]{GuoYe2018}.
\end{remark}

\section{Monotone Proximal Gradient Method}\label{Sec:GenSpecGrad}

We first investigate a monotone version
of the proximal gradient method applied to the composite optimization problem
\eqref{Eq:P} with $ f $ being continuously differentiable and $ \phi $ being lower semicontinuous.
Recall that the corresponding M-stationarity condition is given by
\[
   0 \in \nabla f  (x) + \partial \phi (x).
\]
Our aim is to find, at least approximately, an M-stationary point of \eqref{Eq:P}.
The following algorithm is the classical proximal gradient method for this class of problems.
Since we will also consider a nonmonotone variant of this algorithm in the following
section, we call this the monotone proximal gradient method.

\begin{algorithm}[Monotone Proximal Gradient Method]\leavevmode
	\label{Alg:MonotoneProxGrad}
	\begin{algorithmic}[1]
		\REQUIRE $ \tau > 1, 0 < \gamma_{\min} \leq  \gamma_{\max} < \infty, \delta \in (0,1), x^0 \in \dom\phi $
		\STATE Set $k := 0$.
		\WHILE{A suitable termination criterion is violated at iteration $ k $}
		\STATE Choose $ \gamma_k^0 \in [ \gamma_{\min}, \gamma_{\max}] $.
		\STATE\label{step:subproblem_solve_MonotoneProxGrad} 
			For $ i = 0, 1, 2, \ldots $, compute a solution $ x^{k,i} $ of
      		\begin{equation}\label{Eq:Subki}
         		\min_x \ f (x^k) + \langle\nabla f(x^k), x - x^k \rangle + \frac{\gamma_{k,i}}{2} \| x - x^k \|^2 + \phi (x), 
         		\quad x \in \mathbb X
      		\end{equation}
      		with $ \gamma_{k,i} := \tau^i \gamma_k^0 $, until the acceptance criterion
      		\begin{equation}\label{Eq:StepCrit}
         		\psi (x^{k,i}) \leq 
         		\psi (x^k) - \delta \frac{\gamma_{k,i}}{2} \| x^{k,i} - x^k \|^2 
      		\end{equation}
      		holds.
		\STATE Denote by $ i_k := i $ the terminal value, and set $ \gamma_k := 
      			\gamma_{k,i_k} $ and $ x^{k+1} := x^{k,i_k} $.
      	\STATE Set $ k \leftarrow k + 1 $.
		\ENDWHILE
		\RETURN $x^k$
	\end{algorithmic}
\end{algorithm}

The convergence theory requires some technical assumptions.

\begin{assumption}\label{Ass:ProxGradMonotone}
\begin{enumerate}
   \item \label{item:psi_bounded} The function $ \psi $ is bounded from below on $ \dom\phi $.
   \item \label{item:phi_bounded_affine} The function  $ \phi $ is bounded from below by an affine function.
\end{enumerate}
\end{assumption}

\Cref{Ass:ProxGradMonotone}~\ref{item:psi_bounded} is a reasonable condition regarding the
given composite optimization problem, whereas \Cref{Ass:ProxGradMonotone}~\ref{item:phi_bounded_affine}
is essentially a statement relevant for the subproblems from \eqref{Eq:Subki}. In particular,
\Cref{Ass:ProxGradMonotone}~\ref{item:phi_bounded_affine} implies that the quadratic objective function
of the subproblems \eqref{Eq:Subki} are, for fixed $ k, i \in\N$, coercive, and therefore always
attain a solution $ x^{k,i} $ (which, however, may not be unique).

The subsequent convergence theory assumes implicitly that \Cref{Alg:MonotoneProxGrad}
generates an infinite sequence. 

We first establish that the stepsize rule in \cref{step:subproblem_solve_MonotoneProxGrad} of \cref{Alg:MonotoneProxGrad} is always finite.

\begin{lemma}\label{Lem:StepsizeFinite}
Consider a fixed iteration $ k $ of \cref{Alg:MonotoneProxGrad}, assume that $ x^k $ is not an M-stationary point of
\eqref{Eq:P}, and suppose that \Cref{Ass:ProxGradMonotone}~\ref{item:phi_bounded_affine} holds. Then the inner 
loop in \cref{step:subproblem_solve_MonotoneProxGrad} of \Cref{Alg:MonotoneProxGrad} is finite, i.e., we have $ \gamma_k = \gamma_{k,i_k} $
for some finite index $ i_k \in \{ 0, 1, 2, \ldots \} $.
\end{lemma}

\begin{proof}
	Suppose that the inner loop of \cref{Alg:MonotoneProxGrad} does not
	terminate after a finite number of steps in iteration $k$.
Recall that $ x^{k,i} $ is a solution of \eqref{Eq:Subki}. Therefore, we get
\begin{equation}\label{Eq:Sub1i}
   \langle \nabla f(x^k),x^{k,i} - x^k\rangle +    \frac{\gamma_{k,i}}{2} \| x^{k,i} - x^k \|^2 + \phi (x^{k,i}) \leq \phi (x^k).
\end{equation}
Noting that $ \gamma_{k,i} \to \infty $ for $ i \to \infty $ and using 
\Cref{Ass:ProxGradMonotone}~\ref{item:phi_bounded_affine}, we obtain $ x^{k,i} \to x^k $ for $ i \to \infty $. 
Taking the limit $ i \to \infty $ therefore yields
\begin{equation*}
   \phi (x^k) \leq \liminf_{i \to \infty} \phi (x^{k,i}) \leq 
   \limsup_{i \to \infty} \phi (x^{k,i}) \leq \phi (x^k),
\end{equation*}
where the first estimate follows from the lower semicontinuity of $ \phi $
and the final inequality is a consequence of \eqref{Eq:Sub1i}.
Therefore, we have
\begin{equation}\label{eq:robustsum}
   \phi (x^{k,i}) \to \phi (x^k) \quad
   \text{for } i \to \infty .
\end{equation}
We claim that
\begin{equation}\label{Eq:Sub2i}
   \liminf_{i \to \infty} \gamma_{k,i} \| x^{k,i} - x^k \| > 0.
\end{equation}
Assume, by contradiction, that there is a subsequence $ i_l \to \infty $ such that 
\begin{equation}\label{Eq:Sub3i}
   \liminf_{l \to \infty} \gamma_{k, i_l} \| x^{k, i_l} - x^k \| = 0. 
\end{equation}
Since $ x^{k,i_l} $ is optimal for \eqref{Eq:Subki}, Fermat's rule and the sum rule
\eqref{eq:regular_sum_rule} yield
\begin{equation}\label{eq:optimality_condition_subproblem}
   0 \in  \nabla f (x^k) + \gamma_{k, i_l} ( x^{k, i_l} - x^k ) +
   \widehat\partial \phi (x^{k,i_l})
\end{equation}
for all $l\in\N$.
Taking the limit $ l \to \infty $ while using \eqref{eq:robustsum} and \eqref{Eq:Sub3i},
we obtain
\begin{equation*}
   0 \in  \nabla f (x^k) + \partial \phi (x^k),
\end{equation*}
which means that $ x^k $ is already an M-stationary point of \eqref{Eq:P}. This 
contradiction shows that \eqref{Eq:Sub2i} holds. Hence, there is a constant $ c > 0 $
such that
\begin{equation*}
   \gamma_{k,i} \| x^{k,i} - x^k \| \geq c 
\end{equation*}
holds for all large enough $i\in\N$.
In particular, this implies
\begin{equation}\label{Eq:Sub4i}
   ( 1- \delta ) \frac{\gamma_{k,i}}{2} \| x^{k,i} - x^k \|^2 \geq 
   \frac{1- \delta}{2} c \| x^{k,i} - x^k \| \geq o \big( \| x^{k,i} - x^k \| \big)
\end{equation}
for all sufficiently large $i\in\N$.
Furthermore, \eqref{Eq:Sub1i} shows that
\begin{equation}\label{Eq:Sub5i}
	\langle \nabla f (x^k),x^{k,i} - x^k\rangle +
   \phi \big( x^{k,i} \big) - \phi (x^k)
   \leq
   - \frac{\gamma_{k,i}}{2} \| x^{k,i} - x^k \|^2 .
\end{equation}
Using a Taylor expansion of the 
	function $ f $ 
and exploiting \eqref{Eq:Sub4i},
\eqref{Eq:Sub5i}, we obtain 
\begin{align*}
   \psi ( x^{k,i}) - \psi (x^k)
   & =  
   f ( x^{k,i}) + \phi (x^{k,i}) - f (x^k) - \phi (x^k) \\
   & =  
   \langle \nabla f (x^k),x^{k,i}-x^k\rangle + \phi ( x^{k,i} ) - \phi (x^k) +
   		o \big( \| x^{k,i} - x^k \| \big) \\
   & \leq  
   - \frac{\gamma_{k,i}}{2} \| x^{k,i} - x^k \|^2 + 
   		o \big( \| x^{k,i} - x^k \| \big) \\
   & \leq  
   - \delta \frac{\gamma_{k,i}}{2} \| x^{k,i} - x^k \|^2
\end{align*}
for all $ i \in \N $ sufficiently large.
	This, however, means that the acceptance criterion 
	\eqref{Eq:StepCrit} is valid for sufficiently
	large $i\in\N$,
	contradicting our assumption.
	This completes the proof.
\end{proof}

Let us note that the above proof actually shows that the inner loop from
\cref{step:subproblem_solve_MonotoneProxGrad} of \cref{Alg:MonotoneProxGrad}
is either finite, or we have $\gamma_{k,i_l}\|x^{k,i_l}-x^k\|\to 0$ along
a subsequence $i_l\to\infty$. 
Rewriting \eqref{eq:optimality_condition_subproblem} by means of
\begin{equation}\label{eq:rewritten_optimality_condition_subproblem}
	\nabla f(x^{k,i_l})-\nabla f(x^k)+\gamma_{k,i_l}(x^k-x^{k,i_l})
	\in 
	\nabla f(x^{k,i_l})+\widehat\partial\phi(x^{k,i_l})
\end{equation}
and recalling that $\nabla f\colon\mathbb X\to\mathbb X$ is continuous motivates to also use
\[
	\|\nabla f(x^{k,i})-\nabla f(x^k)+\gamma_{k,i}(x^k-x^{k,i})\|\leq\tau_\textup{abs}
\]
for some $\tau_\textup{abs}>0$ as a termination criterion of the inner loop since this
encodes, in some sense, approximate M-stationarity of $x^{k,i}$ for \eqref{Eq:P}
(note that taking the limit $l\to\infty$ in \eqref{eq:rewritten_optimality_condition_subproblem} 
would recover the limiting subdifferential of $\phi$ at $x^k$
since we have $x^{k,i_l}\to_\phi x^k$ by \eqref{eq:robustsum}).

A critical step for the convergence theory of \Cref{Alg:MonotoneProxGrad} is
provided by the following result.

\begin{proposition}\label{Prop:xdiff}
Let \Cref{Ass:ProxGradMonotone} hold. Then each sequence 
$ \{ x^k \} $ generated by \Cref{Alg:MonotoneProxGrad} satisfies
$ \| x^{k+1} - x^k \| \to 0 $.
\end{proposition}

\begin{proof}
First recall that the sequence $ \{ x^k \} $ is well-defined by \Cref{Lem:StepsizeFinite}.
Using the acceptance criterion \eqref{Eq:StepCrit}, we get
\begin{equation}\label{Eq:xdiff-1}
   \psi (x^{k+1}) \leq \psi (x^k) - \delta \frac{\gamma_k}{2} \| x^{k+1} - x^k \|^2 
   \leq \psi (x^k)
\end{equation}
for all $ k \in \mathbb{{N}} $. Hence, the sequence $ \{ \psi (x^k) \} $ is 
monotonically decreasing. Since $ \psi $ is bounded from below on $ \dom\phi $
by \Cref{Ass:ProxGradMonotone}~\ref{item:psi_bounded} and $ \{ x^k \} \subset \dom\phi $, it follows that
this sequence is convergent. Therefore, \eqref{Eq:xdiff-1} implies
\begin{equation*}
   \gamma_k \| x^{k+1} - x^k \|^2 \to 0 \quad \text{for } k \to \infty .
\end{equation*}
Hence the assertion follows from the fact that, by 
construction, we have $ \gamma_k \geq \gamma_{\min} > 0 $ for all $ k \in \mathbb{N} $.
\end{proof}

A refined analysis gives the following result.

\begin{proposition}\label{Prop:gammaxdiff}
Let \Cref{Ass:ProxGradMonotone} hold, let $\{x^k\}$ be a sequence generated by \cref{Alg:MonotoneProxGrad}, 
and let $ \{ x^k \}_K $ be a subsequence converging
to some point $ x^* $. Then $ \gamma_k \| x^{k+1} - x^k \| \to_K 0 $ holds.
\end{proposition}

\begin{proof}
If the subsequence $ \{ \gamma_k \}_K $ is bounded, the statement follows
immediately from \Cref{Prop:xdiff}. The remaining part of this proof
therefore assumes that this subsequence is unbounded. Without loss of 
generality, we may assume that $ \gamma_k  \to_K \infty $
	and that the acceptance criterion \eqref{Eq:StepCrit} is violated
	in the first iteration of the inner loop for each $k\in K$.
Then, for $ \hat{\gamma}_k := \gamma_k / \tau $, $k\in K$, we also have $\hat\gamma_k\to_K\infty$, but the 
corresponding vector $ \hat x^k := x^{k, i_k-1} $ does not satisfy the stepsize
condition from \eqref{Eq:StepCrit}, i.e., we have
\begin{equation}\label{Prop:gammaxdiff-1}
   \psi (\hat x^k) > 
   \psi (x^k) - \delta \frac{\hat{\gamma}_k}{2} \| \hat{x}^k - x^k \|^2 
   \quad \forall k \in K.
\end{equation}
On the other hand, since $ \hat{x}^k $ solves the corresponding 
subproblem \eqref{Eq:Subki} with $\hat{\gamma}_k = \gamma_{k,i_k-1}$, we have
\begin{equation}\label{Prop:gammaxdiff-2}
    \langle \nabla f (x^k), \hat{x}^k - x^k\rangle + \frac{\hat{\gamma}_k}{2}
   \| \hat{x}^k - x^k \|^2 + \phi ( \hat{x}^k ) - \phi (x^k) \leq 0 .
\end{equation}
We claim that this, in particular, implies $ \hat{x}^k \to_K x^* $. In fact, using
\eqref{Prop:gammaxdiff-2}, the Cauchy-Schwarz inequality, and the monotonicity of
$ \{ \psi (x^k) \} $, we obtain
\begin{align*}
   \frac{\hat{\gamma}_k}{2} \| \hat{x}^k - x^k \|^2 
   & \leq 
   \| \nabla f (x^k) \| \| \hat{x}^k - x^k \| + \phi (x^k) - \phi ( \hat{x}^k ) \\
   & = 
   \| \nabla f (x^k) \| \| \hat{x}^k - x^k \| + \psi (x^k) - f (x^k) - \phi ( \hat{x}^k ) \\
   & \leq 
   \| \nabla f (x^k) \| \| \hat{x}^k - x^k \| + \psi (x^0) - f (x^k) - \phi ( \hat{x}^k ) .
\end{align*}
Since $ f $ is continuously differentiable and $ - \phi $ is bounded from above by 
an affine function in view of \Cref{Ass:ProxGradMonotone}~\ref{item:phi_bounded_affine}, this implies 
$ \| \hat{x}^k - x^k \| \to_K 0 $.
In fact, if $ \{ \| \hat{x}^k - x^k \| \}_K $ would
be unbounded, then the left-hand side would grow more rapidly than the right-hand side,
and if $ \{ \| \hat{x}^k - x^k \| \}_K $ would be bounded, but staying away, at least
on a subsequence, from zero by a positive number, the right-hand side would be bounded, 
whereas the left-hand side would be unbounded on the corresponding subsequence.

Now, by the mean-value theorem, there exists $ \xi^k $ on the line segment connecting
$ x^k $ with $ \hat{x}^k $ such that
\begin{equation}\label{Prop:gammaxdiff-3}
	\begin{aligned}
   		\psi ( \hat{x}^k ) - \psi (x^k) 
   		&= 
   		f ( \hat{x}^k ) + \phi (\hat{x}^k) -   f (x^k) - \phi (x^k) \\
   		&= 
   		\langle  \nabla f (\xi^k),\hat{x}^k - x^k \rangle + \phi (\hat{x}^k) - \phi (x^k) .
   \end{aligned}
\end{equation}
Substituting the expression $ \phi (\hat{x}^k) - \phi (x^k) $ from \eqref{Prop:gammaxdiff-3}
into \eqref{Prop:gammaxdiff-2} yields 
\begin{equation*}
   \langle \nabla f(x^k) - \nabla f ( \xi^k) ,\hat x^k - x^k \rangle +
   \frac{\hat{\gamma}_k}{2} \| \hat{x}^k - x^k \|^2 + \psi ( \hat{x}^k ) - \psi (x^k) \leq 0.
\end{equation*}
Exploiting \eqref{Prop:gammaxdiff-1}, we therefore obtain 
\begin{align*}
   \frac{\hat{\gamma}_k}{2} \| \hat{x}^k - x^k \|^2
   & \leq 
   - \langle \nabla f(x^k) - \nabla f ( \xi^k), \hat x^k - x^k \rangle +  \psi (x^k) - \psi ( \hat{x}^k ) \\
   & \leq
   \| \nabla f (x^k) - \nabla f ( \xi^k) \| \| \hat x^k - x^k \| + \delta \frac{\hat{\gamma}_k}{2} \| \hat{x}^k - x^k \|^2 ,
\end{align*}
which can be rewritten as
\begin{equation}\label{Prop:gammaxdiff-4}
   ( 1 - \delta ) \frac{\hat{\gamma}_k}{2} \| \hat{x}^k - x^k \| \leq 
   \| \nabla f (x^k) - \nabla f ( \xi^k) \| 
\end{equation}
(note that $ \hat{x}^k \neq x^k $ in view of \eqref{Prop:gammaxdiff-1}).
Since $ x^k\to_K x^* $ (by assumption) and $ \hat{x}^k \to_K x^* $ (by the
previous part of this proof), we also get $ \xi^k \to_K x^* $. Using $ \delta \in (0,1) $
and the continuous differentiability of $ f $, it follows from \eqref{Prop:gammaxdiff-4}
that $ \hat{\gamma}_k \| \hat{x}^k - x^k \| \to_K 0 $.

Finally, exploiting the fact that $ x^{k+1} $ and $ \hat{x}^k $ are solutions of
the subproblems \eqref{Eq:Subki} with parameters $ \gamma_k $ and $ \hat{\gamma}_k $,
respectively, we find
\begin{align*}
   \langle \nabla f(x^k),x^{k+1} - x^k\rangle 
   &+
   \frac{\gamma_k}{2} \| x^{k+1} - x^k \|^2 + \phi (x^{k+1})
   \\
   &
   \leq \langle \nabla f (x^k),  \hat{x}^k - x^k \rangle +
   \frac{\gamma_k}{2} \| \hat{x}^k - x^k \|^2 + \phi (\hat{x}^k),
   \\
    \langle \nabla f(x^k), \hat{x}^k - x^k \rangle 
    &+
   \frac{\hat{\gamma}_k}{2} \| \hat{x}^k - x^k \|^2 + \phi (\hat{x}^k)
   \\
   &
   \leq  \langle \nabla f(x^k),  x^{k+1} - x^k \rangle +
   \frac{\hat{\gamma}_k}{2} \| x^{k+1} - x^k \|^2 + \phi (x^{k+1}).
\end{align*}
Adding these two inequalities and noting that $ \gamma_k = \tau \hat{\gamma}_k >
\hat{\gamma}_k $ yields $ \| x^{k+1} - x^k \| \leq \| \hat{x}^k - x^k \| $ and,
therefore,
\begin{equation*}
   \gamma_k \| x^{k+1}- x^k \| = \tau \hat{\gamma}_k \| x^{k+1} - x^k \| \leq
   \tau \hat{\gamma}_k \| \hat{x}^k - x^k \| \to_K 0.
\end{equation*}
This completes the proof.
\end{proof}

The above technique of proof implies a boundedness result for the sequence $ \{ \gamma_k \} _K$ 
if $ \nabla f$ satisfies a local Lipschitz property around the associated accumulation point of iterates. 
This observation is stated explicitly in the  following result.

\begin{corollary}\label{Cor:LipschitzCase}
Let \Cref{Ass:ProxGradMonotone} hold, let $\{x^k\}$ be a sequence generated by \Cref{Alg:MonotoneProxGrad}, let $ \{ x^k \}_K $ be a subsequence converging
to some point $ x^* $, and assume that $ \nabla f \colon \mathbb X\to\mathbb X$ is locally Lipschitz continuous
around $ x^* $. Then the corresponding subsequence $ \{ \gamma_k \}_K $ is
bounded.
\end{corollary}

\begin{proof}
We may argue as in the proof of \cref{Prop:gammaxdiff}. 
Hence, 
	on the contrary, assume
that $ \gamma_k \to_K \infty $. For each $k\in K$, define $ \hat{\gamma}_k $ and $ \hat{x}^k $ as in
that proof, and let $ L > 0 $ denote the local Lipschitz constant of $ \nabla f $ around
$ x^* $. Recall that $ x^k \to_K x^* $ (by assumption) and $ \hat{x}^k \to_K x^* $
(from the proof of \cref{Prop:gammaxdiff}). Exploiting \eqref{Prop:gammaxdiff-4},
we therefore obtain
\begin{equation*}
   ( 1 - \delta ) \frac{\hat{\gamma}_k}{2} \| \hat{x}^k - x^k \| \leq L \| \hat{x}^k - \xi^k \|
   \leq L \| \hat{x}^k - x^k \|
\end{equation*}
for all $ k \in K $ sufficiently large, using the fact that $ \xi^k $ is on the 
line segment between $ x^k $ and $ \hat{x}^k $. Since $ \hat{\gamma}_k \to_K \infty $ and 
$ \hat{x}^k \neq x^k $, see once again \eqref{Prop:gammaxdiff-1}, this gives a contradiction. 
Hence, $ \{ \gamma_k \}_K $ stays bounded.
\end{proof}

The following is the main convergence result for \Cref{Alg:MonotoneProxGrad} which 
requires a slightly stronger smoothness assumption on either $ f $ or $ \phi $.

\begin{theorem}\label{Thm:ConvProxGrad}
Assume that \Cref{Ass:ProxGradMonotone} holds while either $ \phi $ is continuous
on $\dom\phi$ or $ \nabla f\colon\mathbb X\to\mathbb X $ is locally Lipschitz continuous.
Then each accumulation point
$ x^* $ of a sequence $ \{ x^k \} $ generated by \Cref{Alg:MonotoneProxGrad}
is an M-stationary point of \eqref{Eq:P}.
\end{theorem}

\begin{proof}
Let $ \{ x^k \}_K $ be a subsequence converging to $ x^* $. In view of
\Cref{Prop:xdiff}, it follows that also the subsequence $ \{ x^{k+1} \}_K $
converges to $ x^* $. Furthermore, \Cref{Prop:gammaxdiff} yields
$ \gamma_k \| x^{k+1} - x^k \| \to_K 0 $.
The minimizing property of $ x^{k+1} $, Fermat's rule, and the sum rule \eqref{eq:regular_sum_rule} imply that
\begin{equation}\label{Eq:OptCondSub}
   0 \in  \nabla f (x^{k}) + \gamma_k (x^{k+1}-x^k) 
   + \widehat\partial \phi (x^{k+1})
\end{equation}
holds for each $k\in K$. Hence, if we can show $ \phi (x^{k+1}) \to_K \phi (x^*) $, we can
take the limit $ k \to_K \infty $ in \eqref{Eq:OptCondSub} 
to obtain the desired statement $ 0 \in  \nabla f (x^*) + \partial \phi (x^*) $.

	Due to \eqref{Eq:xdiff-1}, we find $\psi(x^{k+1})\leq \psi(x^0)$
	for each $k\in K$. Taking the limit $k\to_K\infty$ while respecting
	the lower semicontinuity of $\phi$ gives $\psi(x^*)\leq\psi(x^0)$,
	and due to $x^0\in\dom\phi$, we find $x^*\in\dom\phi$.
Thus, the condition $ \phi (x^{k+1}) \to_K \phi (x^*) $ obviously holds if
$ \phi $ is continuous on its domain since all iterates $ x^k $ generated by
\Cref{Alg:MonotoneProxGrad} 
	as well as $x^*$
belong to $\dom\phi$.

Hence, it remains to consider the situation where $ \phi $ is only lower semicontinuous, but
$ \nabla f $ is locally Lipschitz continuous. 
From $ x^{k+1} \to_K x^* $ and the lower semicontinuity of $\phi$, we find
\begin{equation*}
   \phi (x^*) \leq \liminf_{k \in K} \phi (x^{k+1}) \leq \limsup_{k \in K} \phi (x^{k+1}).
\end{equation*}
It therefore suffices to show that $ \limsup_{k \in K} \phi (x^{k+1}) \leq \phi (x^*) $ holds.
Since $ x^{k+1} $ solves the subproblem \eqref{Eq:Subki} with parameter $ \gamma_k $, we obtain
\begin{align*}
    \langle \nabla f (x^k),x^{k+1}-x^k \rangle 
    &+ 
    \frac{\gamma_k}{2} \| x^{k+1} - x^k \|^2 +  \phi (x^{k+1} ) 
    \\
    &
    \leq  
    \langle \nabla f(x^k),x^* - x^k \rangle 
   	+ 
   	\frac{\gamma_k}{2} \| x^* - x^k \|^2 + \phi (x^*)
\end{align*}
for each $k\in K$.
We now take the upper limit over $K$ on both sides. 
Using the continuity of $ \nabla f $, 
the convergences $ x^{k+1} - x^k \to_K 0 $ 
as well as $ \gamma_k \| x^{k+1} - x^k \|^2 \to_K 0 $ (see \Cref{Prop:xdiff,Prop:gammaxdiff}),
and taking into account that $ \gamma_k \| x^k - x^* \|^2 \to_K 0 $ due to the boundedness
of the subsequence $ \{ \gamma_k \}_K $ in this situation, see 
\cref{Cor:LipschitzCase}, we obtain 
$ \limsup_{k \in K} \phi (x^{k+1}) \leq \phi (x^*) $. Altogether, we
therefore get $ \phi (x^{k+1}) \to_K \phi (x^*) $, and this completes the proof.
\end{proof}

Note that $ \phi $ being continuous on $ \dom \phi$ is an assumption which holds, e.g., if
$ \phi $ is the indicator function of a closed set, see \Cref{Rem:constrained_opt}. 
Therefore, \Cref{Thm:ConvProxGrad}
provides a global convergence result for constrained optimization problems with an
arbitrary continuously differentiable objective function over any closed (not necessarily
convex) feasible set.
Moreover, the previous convergence
result also holds for a general lower semicontinuous function $ \phi $ provided that $ \nabla f $ 
is locally Lipschitz continuous. This includes, for example, sparse optimization problems
in $\mathbb X\in\{\R^n,\R^{n\times m}\}$ involving the so-called $ \ell_0$-quasi-norm, which counts the number
of nonzero entries of the input vector, as a penalty term or optimization problems in $\mathbb X:=\R^{n\times m}$
comprising rank penalties.
Note that we still do not require the global Lipschitz 
continuity of $ \nabla f $. However, it is an open question whether the previous
convergence result also holds for the general setting where $ f $ is only continuously
differentiable and $ \phi $ is just lower semicontinuous.

\begin{remark}\label{rem:termination_MonotoneGroxGrad}
	Let $\{x^k\}$ be a sequence generated by \cref{Alg:MonotoneProxGrad}.
	In iteration $k\in\N$, $x^{k+1}$ satisfies the necessary optimality
	condition \eqref{Eq:OptCondSub} of the subproblem \eqref{Eq:Subki}.
	Hence, from the next iteration's point of view, we obtain
	\[
		\gamma_{k-1}(x^{k-1}-x^k)+\nabla f(x^k)-\nabla f(x^{k-1})
		\in 
		\nabla f(x^k)+\widehat{\partial}\phi(x^k)
	\]
	for each $k\in\N$ with $k\geq 1$. 
	This justifies evaluation of the termination criterion
	\begin{equation}\label{eq:termination}
		\norm{\gamma_{k-1}(x^{k-1}-x^k)+\nabla f(x^k)-\nabla f(x^{k-1})}\leq\tau_\textup{abs}
	\end{equation}
	for some $\tau_\textup{abs}>0$ since this means that $x^k$ is, in some sense,
	approximately M-stationary for \eqref{Eq:P}.
	Observe that, along a subsequence $\{x^k\}_K$ satisfying $x^{k-1}\to_K x^*$ for some $x^*$,
	\cref{Prop:xdiff,Prop:gammaxdiff} yield $x^k\to_K x^*$ and
	$\gamma_{k-1}(x^k-x^{k-1})\to_K0$ under appropriate assumptions, which means that
	\eqref{eq:termination} is satisfied for large enough $k\in K$ due to continuity
	of $\nabla f\colon\mathbb X\to\mathbb X$,
	see the discussion after \cref{Lem:StepsizeFinite} as well.
\end{remark}

Recall that the existence of accumulation points is guaranteed by the coercivity of
the function $ \psi $. A simple criterion for the convergence of the entire sequence
$ \{ x^k \} $ is provided by the following comment.

\begin{remark}\label{rem:conv_entire_sequence}
Let $ \{ x^k \} $ be any sequence generated by \Cref{Alg:MonotoneProxGrad}
such that $ x^* $ is an isolated accumulation point of this sequence. Then the
entire sequence converges to $ x^* $. This follows immediately from 
\cite[Lemma~4.10]{MoS1983} and the property of the proximal gradient method
stated in \Cref{Prop:xdiff}. The accumulation point $ x^* $ is isolated, in
particular, if $ f $ is twice continuously differentiable with 
$ \nabla^2 f(x^*) $ being positive definite and $ \phi $ is convex. In this
situation, $ x^* $ is a strict local minimum of $ \psi $ and therefore the
only stationary point of $ \psi $ is a neighborhood of $ x^* $. Since,
by \Cref{Thm:ConvProxGrad}, every accumulation point is stationary, it
follows that $ x^* $ is necessarily an isolated stationary point in this
situation and, thus, convergence of the whole sequence $ \{ x^k \} $
to $ x^* $ follows.
\end{remark}

\section{Nonmonotone Proximal Gradient Method}\label{Sec:GenSpecGradNM}

The method to be presented here is a nonmonotone version of the proximal gradient method
from the previous section. The kind of nonmonotonicity used here was introduced by
Grippo et al.\ \cite{GrippoLamparielloLucidi1986} for a class of smooth unconstrained
optimization problems and then discussed, in the framework of composite optimization problems,
by Wright et al.\ \cite{WrightNowakFigueiredo2009} as well as in some subsequent papers.
We first state the precise algorithm and investigate its convergence properties. The relation
to the existing convergence results is postponed until the end of this section.

\begin{algorithm}[Nonmonotone Proximal Gradient Method]\leavevmode
	\label{Alg:NonmonotoneProxGrad}
	\begin{algorithmic}[1]
		\REQUIRE $ \tau > 0$, $0 < \gamma_{\min} \leq \gamma_{\max} < \infty$, $m \in \mathbb{N}$, $\delta\in (0,1)$,   $x^0 \in \dom\phi $
		\STATE Set $k := 0$.
		\WHILE{A suitable termination criterion is violated at iteration $ k $}
		\STATE Set $ m_k := \min \{ k, m \} $ and choose $ \gamma_k^0 \in [ \gamma_{\min}, \gamma_{\max}] $.
		\STATE
			For $ i = 0, 1, 2, \ldots $, compute a solution $ x^{k,i} $ of
	       	\begin{equation}\label{Eq:NonSubki}
         		\min_x \ f (x^k) + \langle \nabla f (x^k),x - x^k\rangle
         		+ \frac{\gamma_{k,i}}{2} \| x - x^k \|^2 + \phi (x),
         		\quad x\in\mathbb X
      		\end{equation}
      		with $ \gamma_{k,i} := \tau^i \gamma_k^0 $, until the acceptance criterion
      		\begin{equation}\label{Eq:NonStepCrit}
         		\psi (x^{k,i}) \leq \max_{j=0, 1, \ldots, m_k}
         		\psi (x^{k-j}) - \delta \frac{\gamma_{k,i}}{2} \| x^{k,i} - x^k \|^2 
      		\end{equation}
      		holds.
		\STATE Denote by $ i_k := i $ the terminal value, and set $ \gamma_k := 
      			\gamma_{k,i_k} $ and $ x^{k+1} := x^{k,i_k} $.
      	\STATE Set $ k \leftarrow k + 1 $.
		\ENDWHILE
		\RETURN $x^k$
	\end{algorithmic}
\end{algorithm}

The only difference between \Cref{Alg:MonotoneProxGrad} and \Cref{Alg:NonmonotoneProxGrad}
is in the stepsize rule. More precisely, \Cref{Alg:NonmonotoneProxGrad} may be viewed
as a generalization of \Cref{Alg:MonotoneProxGrad} since the particular choice $ m = 0 $
recovers \Cref{Alg:MonotoneProxGrad}. Numerically, in many examples, the choice $ m > 0 $
leads to better results and is therefore preferred in practice. On the other hand, for 
$ m > 0 $, we usually get a nonmonotone behavior of the function values $ \{\psi (x^k) \} $
which complicates the theory significantly. 
In addition, the nonmontone proximal
gradient method also requires stronger assumptions in order to prove a suitable
convergence result.

In particular, in addition to the requirements from \cref{Ass:ProxGradMonotone}, we need 
the following additional conditions on the data functions in order to proceed.

\begin{assumption}\label{Ass:ProxGradNonmonotone}
	\begin{enumerate}
		\item\label{item:uniform_continuity}
			The function $ \psi $ is uniformly continuous on the sublevel set 
			$\mathcal L_\psi(x^0):=\{x\in \mathbb X\,|\,\psi(x)\leq\psi(x^0)\}$.
		\item\label{item:continuity_phi}
			The function $\phi$ is continuous on $\dom\phi$.
	\end{enumerate}
\end{assumption}

Note that we always have $\mathcal L_\psi(x^0)\subset\dom\phi$ by the continuity of $f$.
Furthermore, whenever $\psi$ is coercive, \cref{Ass:ProxGradNonmonotone}~\ref{item:continuity_phi}
already implies \cref{Ass:ProxGradNonmonotone}~\ref{item:uniform_continuity} since
$\mathcal L_\psi(x^0)$ would be a compact subset of $\dom\phi$ in this situation, and continuous
functions are uniformly continuous on compact sets. Observe that coercivity of $\psi$ is an
inherent property in many practically relevant settings. We further note that, in general,
\cref{Ass:ProxGradNonmonotone}~\ref{item:uniform_continuity} does not imply
\cref{Ass:ProxGradNonmonotone}~\ref{item:continuity_phi}, and the latter is a necessary
requirement since, in our convergence
theory, we will also evaluate the function $ \phi $ in some points resulting from an auxiliary
sequence $ \{ \hat{x}^k \} $ which may not belong to the level set $\mathcal L_\psi(x^0) $.

For the convergence theory, we assume implicitly that 
\cref{Alg:NonmonotoneProxGrad} generates an infinite sequence  $\{x^k\}$. 
We first
note that the stepsize rule in the inner loop of \cref{Alg:NonmonotoneProxGrad} is always finite.
Since 
\[
	\psi(x^k)\leq\max\limits_{j=0,1,\ldots,m_k}\psi(x^{k-j})
\]
this observation follows immediately from \Cref{Lem:StepsizeFinite}.

Throughout the section, for each $k\in\N$, let $l(k)\in\{k-m_k,\ldots,k\}$ be an index such that
\[
	\psi(x^{l(k)})=\max\limits_{j=0,1,\ldots,m_k}\psi(x^{k-j})
\]
is valid. We already mentioned that $\{\psi(x^k)\}$ may possess a nonmonotone behavior.
However, as the following lemma shows, $\{\psi(x^{l(k)})\}$ is monotonically decreasing.

\begin{lemma}\label{Lem:decrease_condition_in_nonmonotone_framework}
	Let \cref{Ass:ProxGradMonotone}~\ref{item:phi_bounded_affine} hold
	and let $\{x^k\}$ be a sequence generated by \cref{Alg:NonmonotoneProxGrad}.
	Then $\{\psi(x^{l(k)})\}$ is monotonically decreasing.
\end{lemma}

\begin{proof}
	The nonmonotone stepsize rule from \eqref{Eq:NonStepCrit} can be rewritten as
	\begin{equation}\label{Eq:NA1}
  		\psi(x^{k+1}) \leq \psi(x^{l(k)}) - \delta
  		\frac{\gamma_k}{2} \| x^{k+1} - x^k \|^2.
	\end{equation}
	Using $ m_{k+1} \leq m_k + 1 $, we find
	\begin{align*}
   		\psi ( x^{l(k+1)}) & 
   		= \max_{j=0,1, \ldots, m_{k+1}}\psi (x^{k+1-j}) \\
   		& \leq \max_{j=0,1, \ldots, m_k +1}\psi (x^{k+1-j}) \\
   		& = \max \left\{ \max_{j = 0, 1, \ldots, m_k}\psi (x^{k-j}),
   		\psi (x^{k+1}) \right\} \\
   		& = \max \left\{ \psi (x^{l(k)}),\psi (x^{k+1}) \right\} \\
   		& = \psi (x^{l(k)}),
	\end{align*}
	where the last equality follows from \eqref{Eq:NA1}.
	This shows the claim.
\end{proof}

As a corollary of the above result, we obtain that the iterates of \cref{Alg:NonmonotoneProxGrad}
belong to the level set $\mathcal L_\psi(x^0)$.

\begin{corollary}\label{Cor:level_set_condition_nonmonotone_framework}
	Let \cref{Ass:ProxGradMonotone}~\ref{item:phi_bounded_affine} hold
	and let $\{x^k\}$ be a sequence generated by \cref{Alg:NonmonotoneProxGrad}.
	Then $\{x^k\},\{x^{l(k)}\}\subset\mathcal L_\psi(x^0)$ holds.
\end{corollary}

\begin{proof}
	Noting that $l(0)=0$ holds by construction, \cref{Lem:decrease_condition_in_nonmonotone_framework} 
	and \eqref{Eq:NA1} yield the estimate $\psi(x^{k+1})\leq\psi(x^{l(k)})\leq\psi(x^{l(0)})=\psi(x^0)$ for each $k\in\N$
	which shows the claim.
\end{proof}

The counterpart of \Cref{Prop:xdiff} is significantly more difficult to prove in the nonmonotone setting.
In fact, it is this central result which requires the uniform continuity of
the objective function $ \psi $ from \cref{Ass:ProxGradNonmonotone}~\ref{item:uniform_continuity}. Though its proof is essentially the one
from \cite{WrightNowakFigueiredo2009}, we present all details since they 
turn out to be of some importance for the discussion at the end of this section.

\begin{proposition}\label{Prop:xdiff_NM}
	Let \Cref{Ass:ProxGradMonotone} and \cref{Ass:ProxGradNonmonotone}~\ref{item:uniform_continuity} hold.
	Then each sequence $\{x^k\}$ generated by \cref{Alg:NonmonotoneProxGrad} satisfies $ \| x^{k+1} - x^k \| \to 0 $.
\end{proposition}

\begin{proof}
 Since $ \psi $ is bounded from 
below due to \cref{Ass:ProxGradMonotone}~\ref{item:psi_bounded}, \cref{Lem:decrease_condition_in_nonmonotone_framework} implies 
\begin{equation}\label{Eq:Limitfk}
   \lim_{k \to \infty}\psi ( x^{l(k)}) 
   = \psi^*
\end{equation}
for some finite $\psi^*\in\R$.
From \cref{Cor:level_set_condition_nonmonotone_framework}, we find $\{x^{l(k)}\}\subset\mathcal L_\psi(x^0)$. 
Applying \eqref{Eq:NA1} with $k$ replaced by $l(k)-n-1$ for some $n\in\N$ gives $\psi(x^{l(k)-n})\leq\psi(x^{l(l(k)-n-1)})\leq\psi(x^0)$, i.e.,
$\{x^{l(k)-n}\}\subset\mathcal L_\psi(x^0)$ (here, we assume implicitly that $k$ is large enough such that no negative indices $l(k)-n-1$ occur).
More precisely, for $n=0$, we have
\begin{equation*}
   \psi (x^{l(k)}) -\psi (x^{l(l(k)-1)}) 
   \leq - \delta \frac{\gamma_{l(k)-1}}{2}
   \| x^{l(k)} - x^{l(k)-1} \|^2 \leq 0.
\end{equation*}
Taking the limit $ k \to \infty $ in the previous inequality and using \eqref{Eq:Limitfk},
we therefore obtain
\begin{equation*}
   \lim_{k \to \infty} \gamma_{l(k)-1} \| x^{l(k)} - x^{l(k)-1} \|^2 = 0 .
\end{equation*}
Since $ \gamma_k \geq \gamma_{\min} > 0 $ for all $ k \in \mathbb{N} $, we get
\begin{equation}\label{Eq:Ind1}
   \lim_{k \to \infty} d^{l(k)-1} = 0,
\end{equation}
where $ d^k := x^{k+1} - x^k $ for all $ k \in \mathbb{N} $. 
Using \eqref{Eq:Limitfk} and \eqref{Eq:Ind1}, it follows that
\begin{equation}\label{Eq:Ind2}
   \psi^* = \lim_{k \to \infty}\psi (x^{l(k)}) =
   \lim_{k \to \infty} \psi \big( x^{l(k)-1} + d^{l(k)-1} \big) =
   \lim_{k \to \infty} \psi (x^{l(k)-1}),
\end{equation}
where the last equality takes into account the uniform continuity of $ \psi $ 
from \cref{Ass:ProxGradNonmonotone}~\ref{item:uniform_continuity}
and \eqref{Eq:Ind1}.

We will now prove, by induction, that the limits
\begin{equation}\label{Eq:Indj}
   \lim_{k \to \infty} d^{l(k)-j} = 0,\qquad
   \lim_{k \to \infty}\psi (x^{l(k)-j}) = \psi^* 
\end{equation}
hold for all $ j \in\N$ with $j\geq 1$.
We already know from \eqref{Eq:Ind1} and \eqref{Eq:Ind2} that \eqref{Eq:Indj} holds for 
$ j = 1 $. Suppose that \eqref{Eq:Indj} holds for some $ j \geq 1 $. We need to show that
it holds for $ j+1 $. Using \eqref{Eq:NA1} with $ k $ replaced by
$ l(k)-j-1 $, we have
\begin{equation*}
  \psi (x^{l(k)-j}) \leq \psi (x^{l(l(k)-j-1)}) - \delta 
  \frac{\gamma_{l(k)-j-1}}{2} \| d^{l(k)-j-1} \|^2 
\end{equation*}
(again, we assume implicitly that $ k $ is large enough such that $ l(k)-j-1 $ is nonnegative). 
Rearranging this expression and using $ \gamma_k \geq
\gamma_{\min} $ for all $ k $ yields
\begin{equation*}
   \| d^{l(k)-j-1} \|^2 
   \leq \frac{2}{\gamma_{\min} \delta}
   \big(\psi(x^{l(l(k)-j-1)}) -\psi (x^{l(k)-j}) \big).
\end{equation*}
Taking $ k \to \infty $, using \eqref{Eq:Limitfk}, as well as the induction
hypothesis, it follows that 
\begin{equation}\label{Eq:dstar}
   \lim_{k \to \infty} d^{l(k)-j-1} = 0,
\end{equation}
which proves the induction step for the first limit in \eqref{Eq:Indj}.
The second limit then follows from 
\begin{equation*}
   \lim_{k \to \infty} \psi \big( x^{l(k)-(j+1)} \big) =
   \lim_{k \to \infty} \psi \big( x^{l(k)-(j+1)} + d^{l(k)-j-1)} \big) =
   \lim_{k \to \infty} \psi \big( x^{l(k)-j} \big) =
   \psi^*,
\end{equation*}
where the first equation exploits \eqref{Eq:dstar} together with
the uniform continuity of $ \psi $ from \cref{Ass:ProxGradNonmonotone}~\ref{item:uniform_continuity} 
and $\{x^{l(k)-j}\},\{x^{l(k)-(j+1)}\}\subset\mathcal L_\psi(x^0)$, 
whereas the final equation is the induction hypothesis.

In the last step of our proof, we now show that $ \lim_{k \to \infty} d^k = 0 $
holds. Suppose that this is not true. Then there is a (suitably shifted, for notational
simplicity) subsequence $ \{ d^{k-m-1} \}_{k \in K} $
and a constant $ c > 0 $ such that
\begin{equation}\label{Eq:Contrad}
   \| d^{k-m-1} \| \geq c \quad \forall k \in K.
\end{equation}
Now, for each $ k \in K $, the corresponding index $ l(k) $ is one of the indices
$ k - m, k - m + 1, \ldots, k $. Hence, we can write $ k - m - 1 = l(k) - j_k $
for some index $ j_k \in \{ 1, 2, \ldots, m+1 \} $. Since there are only finitely
many possible indices $ j_k $, we may assume without loss of generality that
$ j_k = j $ holds for some fixed index $ j \in \{1,\ldots,m+1\}$. Then \eqref{Eq:Indj} implies
\begin{equation*}
   \lim_{k \to_K \infty} d^{k-m-1} = \lim_{k \to_K \infty} d^{l(k) - j} = 0.
\end{equation*}
This contradicts \eqref{Eq:Contrad} and therefore completes the proof.
\end{proof}

\begin{theorem}\label{Thm:ConvProxGradNM}
Assume that \Cref{Ass:ProxGradMonotone,Ass:ProxGradNonmonotone} hold
and let $\{x^k\}$ be a sequence generated by \cref{Alg:NonmonotoneProxGrad}.
Suppose that $x^*$ is an accumulation point of $\{x^k\}$ such that
$x^k\to_K x^*$ holds along a subsequence $k\to_K\infty$.
Then $ x^* $ is an M-stationary point of \eqref{Eq:P}, and $\gamma_k(x^{k+1}-x^k)\to_K 0$ is valid.
\end{theorem}

\begin{proof}
Since $ \{ x^k \}_K $ is a subsequence converging to $ x^* $, it follows from
\cref{Prop:xdiff_NM} that also the subsequence $ \{ x^{k+1} \}_K $
converges to $ x^* $.
	We note that $x^*\in\dom\phi$ follows from
	\cref{Cor:level_set_condition_nonmonotone_framework}
	by closedness of $\mathcal L_\psi(x^0)$.
The minimizing property of $ x^{k+1} $ for \eqref{Eq:NonSubki} together
with Fermat's rule and the sum rule from \eqref{eq:regular_sum_rule} imply that 
the necessary optimality condition \eqref{Eq:OptCondSub}
holds for each $k\in K$. We claim that the subsequence $ \{ \gamma_k \}_K $ is bounded. Assume, by contradiction,
that this is not true. Without loss of generality, let us assume that $ \gamma_k \to_K \infty $
	and that the acceptance criterion \eqref{Eq:NonStepCrit} is violated in the first
	iteration of the inner loop for each $k\in K$.
Setting $ \hat{\gamma}_k := \gamma_k / \tau $ for each $k\in K$, $\{\hat\gamma^k\}_K$ also tends to infinity, but the 
corresponding vectors $ \hat x^k := x^{k, i_k-1} $, $k\in K$, do not satisfy the stepsize
condition from \eqref{Eq:NonStepCrit}, i.e., we have
\begin{equation}\label{Eq:ContraStepSub}
   \psi (\hat x^k) > \max_{j=0, 1, \ldots, m_k}
   \psi (x^{k-j}) - \delta \frac{\hat{\gamma}_k}{2} \| \hat{x}^k - x^k \|^2 
   \qquad\forall k\in K.
\end{equation}
On the other hand, since $ \hat{x}^k = x^{k,i_k-1} $ solves the corresponding 
subproblem \eqref{Eq:Subki} with $ \hat{\gamma}_k = \gamma_{k, i_k-1} $, we have
\begin{equation}\label{Eq:optimality_subproblem_previous_inner_iteration}
   \langle \nabla f (x^k),\hat{x}^k - x^k\rangle + \frac{\hat{\gamma}_k}{2}
   \| \hat{x}^k - x^k \|^2 + \phi ( \hat{x}^k ) \leq \phi (x^k) 
\end{equation}
for each $k\in K$.
Due to $ \hat{\gamma}_k \to_K \infty $ and since $ \phi $ is bounded from below by an affine
function due to \cref{Ass:ProxGradMonotone}~\ref{item:phi_bounded_affine} while $ \phi $ 
is continuous on its domain by \cref{Ass:ProxGradNonmonotone}~\ref{item:continuity_phi}
(which yields boundedness of the right-hand side of \eqref{Eq:optimality_subproblem_previous_inner_iteration}), 
this implies
$ \hat x^k - x^k \to_K 0 $. Consequently, we have $\hat{x}^k\to_K  x^* $ as well.

Now, if $\hat\gamma_k\|\hat x^k-x^k\|\to_{K'} 0$ holds along a subsequence 
	$k\to_{K'}\infty$ such that $K'\subset K$,
then, due to
\[
   0\in \nabla f(x^k)+\hat\gamma_k(\hat x^k-x^k)+\widehat\partial\phi(\hat x^k),
\] 
which holds for each $k\in K'$ by means of Fermat's rule and the sum rule \eqref{eq:regular_sum_rule},
we immediately see that $x^*$ is an M-stationary point of \eqref{Eq:P} by taking the limit $k\to_{K'}\infty$ 
and exploiting the continuity of $\phi$ on $\dom\phi$ from \cref{Ass:ProxGradNonmonotone}~\ref{item:continuity_phi}. 
Thus, for the remainder of the proof, we may assume 
that there is a constant $c>0$ such that 
\begin{equation*}
   \hat{\gamma}_k \| \hat{x}^k - x^k \| \geq c
\end{equation*}
holds for each $k\in K$. Further, we then also get 
\begin{equation*}
   ( 1- \delta ) \frac{\hat{\gamma}_k}{2} \| \hat{x}^k - x^k \|^2 \geq
   \frac{1- \delta}{2} c \| \hat{x}^k - x^k \| \geq o \big( \| \hat{x}^k - x^k \| \big)
\end{equation*}
for all $ k \in K $ sufficiently large. 
Rearranging \eqref{Eq:optimality_subproblem_previous_inner_iteration} gives us
\[
   \langle\nabla f(x^k) ,\hat{x}^k - x^k \rangle + \phi (\hat{x}^k) - \phi (x^k)
   \leq
   -\frac{\hat{\gamma}^k}{2} \| \hat{x}^k- x^k \|^2 
\]
for each $k\in K$.
From the mean-value theorem, we obtain some $\xi^k$ on the line segment between $\hat x^k$ and $x^k$ such that
\begin{align*}
   &\psi ( \hat{x}^k) - \max_{j= 0, 1, \ldots, m_k} \psi (x^{k-j}) \\
   &\qquad 
   \leq 
   \psi ( \hat{x}^k) - \psi (x^k) \\
   &\qquad 
   =  
   \langle \nabla f( \xi^k) , \hat{x}^k - x^k \rangle + \phi ( \hat{x}^k ) - \phi (x^k) \\
   & \qquad
   =  
   \langle \nabla f(x^k) , \hat{x}^k - x^k \rangle + \phi ( \hat{x}^k ) - \phi (x^k) +
   \langle \nabla f ( \xi^k) - \nabla f(x^k) , \hat{x}^k - x^k \rangle \\
   & \qquad
   \leq  
   - \frac{\hat{\gamma}^k}{2} \| \hat{x}^k- x^k \|^2 + o (\| \hat{x}^k - x^k \|) \\
   &\qquad 
   \leq  
   - \delta \frac{\hat{\gamma}^k}{2} \| \hat{x}^k- x^k \|^2
\end{align*}
for all $ k \in K $ sufficiently large. This contradiction to \eqref{Eq:ContraStepSub}
shows that the sequence $ \{ \gamma_k \}_K $ is bounded.

Finally, the continuity of $\phi$ from \cref{Ass:ProxGradNonmonotone}~\ref{item:continuity_phi} gives
$\phi(x^{k+1})\to_K\phi(x^*)$ due to $x^{k+1}\to_K x^*$.
Thus, recalling $x^k\to_K x^*$ and the boundedness of $\{\gamma_k\}_K$,
we find $\gamma_k(x^{k+1}-x^k)\to_K 0$, and
taking the limit $k\to_K\infty$ in \eqref{Eq:OptCondSub} gives us 
M-stationarity of $x^*$ for \eqref{Eq:P}.
\end{proof}

\begin{remark}\label{rem:remarks_regarding_ProxGradNM}
\begin{enumerate}
	\item Note that \cref{Ass:ProxGradMonotone,Ass:ProxGradNonmonotone} do not comprise any Lipschitz
		conditions on $\nabla f$.
	\item The results in this section recover the findings from \cite[Section~4]{GuoDeng2021} and \cite[Section~3]{JiaKanzowMehlitzWachsmuth2021}
		which were obtained in the special situation where $\phi$ is the indicator function associated with a closed set,
		see \cref{Rem:constrained_opt} as well.
	\item Based on \cref{Thm:ConvProxGradNM}, \eqref{eq:termination} also provides a reasonable
		termination criterion for \cref{Alg:NonmonotoneProxGrad}, see \cref{rem:termination_MonotoneGroxGrad}
		as well.
    \item In view of \cref{Prop:xdiff_NM}, it follows in the same way as in 
        \cref{rem:conv_entire_sequence} that the entire sequence $ \{ x^k \} $ generated
        by \cref{Alg:NonmonotoneProxGrad} converges if there exists an isolated 
        accumulation point.
\end{enumerate}
\end{remark}

The uniform continuity of $ \psi $ which is demanded in \cref{Ass:ProxGradNonmonotone}~\ref{item:uniform_continuity} 
is obviously a much stronger assumption than the 
one used in the previous section for the monotone proximal gradient method. In particular,
this assumption rules out applications where $ \phi $ is given by the $ \ell_0 $-quasi-norm. Nevertheless,
the theory still covers the situation where the role of $ \phi $ 
is played by an $\ell_p$-type penalty function for $p\in(0,1)$ over $\mathbb X\in\{\R^n,\R^{n\times m}\}$
which is known to promote sparse solutions.
More precisely, this choice is popular in sparse optimization if
the more common $ \ell_1 $-norm does not provide satisfactory sparsity results, and 
the application of the $ \ell_0 $-quasi-norm seems too difficult,
see \cite{BianChen2015,Chartrand2007,ChenGuoLuYe2017,DeMarchiJiaKanzowMehlitz2022,LiuDaiMa2015,MarjanovicSolo2012} for some applications and numerical results
based on the $ \ell_p $-quasi-norm or closely related expressions.
We would like to note that uniform continuity is a 
standard assumption in the context of nonmonotone stepsize rules involving acceptance criteria of type \eqref{Eq:NonStepCrit},
see \cite[page 710]{GrippoLamparielloLucidi1986}.

We close this section with a discussion on existing convergence results for nonmonotone
proximal gradient methods. To the best of our knowledge, the first one can be found 
in \cite{WrightNowakFigueiredo2009}.
The authors prove convergence under the assumptions that $ f $ is differentiable with
a globally Lipschitz continuous gradient and $ \phi $ being real-valued and convex,
see \cite[Section~II.G]{WrightNowakFigueiredo2009}.
Implicitly, however, they also exploit the uniform continuity of $ \psi = f + \phi $ in
their proof of \cite[Lemma~4]{WrightNowakFigueiredo2009}, a result like \Cref{Prop:xdiff_NM}, without stating this assumption explicitly.
Taking this into account, our \Cref{Ass:ProxGradNonmonotone}~\ref{item:uniform_continuity} is actually weaker than
the requirements used in \cite{WrightNowakFigueiredo2009}, so that the results of this section can be viewed
as a generalization of the convergence theory from \cite{WrightNowakFigueiredo2009}.

Furthermore, \cite[Section~3.1]{ChenGuoLuYe2017} and \cite[Appendix~A]{ChenLuPong2016} consider a nonmonotone proximal gradient method
which is slightly different from \Cref{Alg:NonmonotoneProxGrad} since the acceptance criterion \eqref{Eq:NonStepCrit} is replaced by
the slightly simpler condition
\[
         \psi (x^{k,i}) \leq \max_{j=0, 1, \ldots, m_k}
         \psi (x^{k-j}) - \frac{\delta}{2} \| x^{k,i} - x^k \|^2.
\]
In \cite[Theorem~4.1]{ChenLuPong2016}, the authors obtain convergence
to M-stationary points whenever $\psi$ is bounded from below as well as uniformly continuous on the level set $\mathcal L_\psi(x^0)$,
$f$ possesses a Lipschitzian derivative on some enlargement of $\mathcal L_\psi(x^0)$, and $\phi$ is continuous.
Clearly, our convergence analysis of \cref{Alg:NonmonotoneProxGrad} does not exploit any Lipschitzianity of $\nabla f$, so our assumptions
are weaker than those ones used in \cite{ChenLuPong2016}. 
In \cite[Theorem~3.3]{ChenGuoLuYe2017}, the authors claim that the results from \cite{ChenLuPong2016} even hold when 
the continuity assumption on $\phi$ is dropped. The proof of \cite[Theorem~3.3]{ChenGuoLuYe2017}, however, 
relies on the outer semicontinuity property \eqref{Eq:osc} of the limiting subdifferential, which does not hold
for general discontinuous functions $ \phi $, so this result is not reliable.  

Finally, let us mention that the two references \cite{LiLin2015,WangLiu2021} also 
consider nonmonotone (and accelerated) proximal gradient methods. These methods are not
directly comparable to our algorithm since they are based on a different kind of 
nonmonotonicity. In any case, although the analysis in both papers works for
merely lower semicontinuous functions $\phi$, the provided convergence theory requires 
$ \nabla f $ to be globally Lipschitz continuous.

\section{Conclusions}\label{Sec:Final}

In this paper, we demonstrated how the convergence analysis for monotone and nonmonotone proximal gradient
methods can be carried out in the absence of (global) Lipschitz continuity of the derivative associated with
the smooth function. Our results, thus, open up these algorithms to be reasonable candidates
for subproblem solvers within an augmented Lagrangian framework for the numerical treatment of constrained
optimization problems with lower semicontinuous objective functions, see e.g.\ \cite{ChenGuoLuYe2017}
where this approach has been suggested but suffers from an incomplete analysis,
and \cite{GuoDeng2021,DeMarchiJiaKanzowMehlitz2022,JiaKanzowMehlitzWachsmuth2021} where this approach has
been corrected and extended.

Let us mention some remaining open problems regarding the investigated proximal gradient methods.
First, it might be interesting to find minimum requirements which ensure global convergence of
\cref{Alg:MonotoneProxGrad,Alg:NonmonotoneProxGrad}. We already mentioned in \cref{Sec:GenSpecGrad} that
it is an open question whether the convergence analysis for \cref{Alg:MonotoneProxGrad} can be generalized
to the setting where $f$ is only continuously differentiable while $\phi$ is just lower semicontinuous.
Second, we did not investigate if the \emph{Kurdyka--\L ojasiewicz} property could be efficiently
incorporated into the convergence analysis in order to get stronger results even in the absence of strong
Lipschitz assumptions on the derivative of $f$. 
Third, our analysis has shown that \cref{Alg:MonotoneProxGrad,Alg:NonmonotoneProxGrad} compute M-stationary
points of \eqref{Eq:P} in general. In the setting of \cref{rem:non_Lipschitz_constrained_optimization}, i.e., where constrained
programs with a merely lower semicontinuous objective function are considered, the introduced concept of M-stationarity 
is, to some extent, \emph{implicit} since it comprises an unknown subdifferential.
In general, the latter can be approximated from above in terms of initial problem data only in situations
where a qualification condition is valid. The resulting stationarity condition may be referred to as \emph{explicit}
M-stationarity. It seems to be a relevant topic of future research to investigate
whether \cref{Alg:MonotoneProxGrad,Alg:NonmonotoneProxGrad} can be modified such that they compute explicitly
M-stationary points in this rather general setting.
	Fourth, it might be interesting to investigate whether other types of nonmonotonicity, different from the one used
	in \cref{Alg:NonmonotoneProxGrad}, can be exploited in order to get rid of the uniform continuity requirement from
	\cref{Ass:ProxGradNonmonotone}\,\ref{item:uniform_continuity}.

Finally, we note that there exist several generalizations of proximal gradient methods using, e.g.,
inertial terms and Bregman distances, 
see e.g.\ \cite{BauschkeBolteTeboulle2017,BolteSabachTeboulleVaisbourd2018,BotCsetnek2016,BotCsetnekLaszlo2016} and the references
therein. The corresponding 
convergence theory is also based on a global Lipschitz assumption for the gradient of
the smooth term or additional convexity assumptions which allow the application of a 
descent-type lemma. It might be interesting to see whether our technique of
proof can be adapted to these generalized proximal gradient methods in order to weaken
the postulated assumptions.

%%%%%%% Bibliography
%%\bibliographystyle{plainnat}
%\bibliographystyle{habbrv}
%\bibliography{references}

\end{document}